\documentclass[a4paper]{article}
\usepackage[T1]{fontenc}
\usepackage{mathptmx}
\usepackage{amsmath,amssymb,amsthm,amsbsy}
\usepackage{mathrsfs}
\usepackage{graphics,graphicx,color}
\usepackage{hyperref}
\usepackage{enumerate}
\usepackage{xcolor}
\usepackage{orcidlink}
\usepackage{setspace}
\usepackage[pagewise,mathlines]{lineno} 
\newtheorem{lem}{Lemma}[section]
\newtheorem{thm}[lem]{Theorem}
\newtheorem{cor}[lem]{Corollary}
\newtheorem{prop}[lem]{Proposition}

\newtheorem{rem}[lem]{Remark}

\newcommand{\im}{\mathrm{im}}

\def\ker{\mathop{\rm ker}}

\def\im{\mathop{\rm im}}

\newcommand{\rank}{\mathrm{rank}}

\newcommand{\reg}{\mathrm{Reg}}

\begin{document}
\begin{center}
\textbf{\large{Transformation Semigroups Which Are Disjoint Union of Symmetric Groups}}\\%\footnote{Part of this work was presented in the Arbeitstagung Allgemeine Algebra conference (AAA93), Bern, Switzerland, February 10-12, 2017.}}\\
\vspace{0.5 cm} Utsithon Chaichompoo and Kritsada Sangkhanan \orcidlink{0000-0002-1909-7514}\footnote{Corresponding author.} 
\end{center}

%%%%%%%%%%%%%%%%%%%%%%%%%%%abstract%%%%%%%%%%%%%%%%%%%%%%%%%%%%%%%%%%%%%%%%%%%%%%%%%%%%%%%%%%%%%%%%%%%%%%%%%%%%%%%%%%%%%%%%%%%%%

\begin{abstract}
	Let $X$ be a nonempty set and $T(X)$ the full transformation semigroup on $X$. For any equivalence relation $E$ on $X$, define a subsemigroup $T_{E^*}(X)$ of $T(X)$ by
	$$
	T_{E^*}(X)=\{\alpha\in T(X):\text{for all}\ x,y\in X, (x,y)\in E\Leftrightarrow (x\alpha,y\alpha)\in E\}.
	$$
	We have the regular part of $T_{E^*}(X)$, denoted by $\reg(T)$, is the largest regular subsemigroup of $T_{E^*}(X)$. Defined the subsemigroup $Q_{E^*}(X)$ of $T_{E^*}(X)$ by
	$$
	Q_{E^*}(X)=\{\alpha\in T_{E^*}(X):|A\alpha|=1\ \text{and}\ A\cap X\alpha\neq\emptyset\ \text{for all}\ A\in X/E\}.
	$$
	Then we can prove that this subsemigroup is the (unique) minimal ideal of $\reg(T)$ which is called the kernel of $\reg(T)$. In this paper, we will compute the rank of $Q_{E^*}(X)$ when $X$ is finite and prove an isomorphism theorem. Finally, we describe and count all maximal subsemigroups of $Q_{E^*}(X)$ where $X$ is a finite set.
\end{abstract}

%%%%%%%%%%%%%%%%%%%%%%%%%%%Introduction%%%%%%%%%%%%%%%%%%%%%%%%%%%%%%%%%%%%%%%%%%%%%%%%%%%%%%%%%%%%%%%%%%%%%%%%%%%%%%%%%%%%%%%%%%%%%

\section{Introduction}

The set of all functions from a set $X$ into itself, denoted as $T(X)$, forms a regular semigroup under the composition of functions. This semigroup is known as the \textit{full transformation semigroup on X} and is important in algebraic semigroup theory. Similar to Cayley's Theorem for groups, it can be shown that any semigroup $S$ can be embedded in the full transformation semigroup $T(S^1)$ where $S^1$ is a monoid obtained from $S$ by adjoining an identity if necessary.

For an equivalence relation $E$ on a set $X$, in 2005, H. Pei \cite{Pei} introduced a new semigroup called the \textit{transformation semigroup that preserves equivalence}, $T_E(X)$, defined by
$$
T_{E}(X)=\{\alpha\in T(X):\forall x,y\in X,(x,y)\in E\Rightarrow (x\alpha,y\alpha)\in E\}.
$$ 
This semigroup is a generalization of the full transformation semigroup $T(X)$ and is defined  as a subsemigroup of $T(X)$. Clearly, if $E=X\times X$, then $T_E(X)$ is equal to $T(X)$. In the field of topology, $T_E(X)$ refers to the semigroup comprising all continuous self-maps of a given topological space $X$ that satisfy the condition of having the $E$-classes serve as a basis. This semigroup is commonly denoted as $S(X)$. The regularity of elements and Green's relations for $T_E(X)$ were investigated in \cite{Pei}. In 2008, L. Sun, H. Pei, and Z. Cheng \cite{Sun} characterized the natural partial order $\leq$ on $T_E(X)$ and studied its compatibility. They also described the maximal, minimal, and covering elements of $T_E(X)$.

Later, in 2010, L. Deng, J. Zeng, and B. Xu \cite{Deng} introduced a new semigroup called the \textit{transformation semigroup that preserve double direction equivalence}, $T_{E^*}(X)$, defined by
$$
T_{E^*}(X)=\{\alpha\in T(X):\forall x,y\in X,(x,y)\in E\Leftrightarrow (x\alpha,y\alpha)\in E\}.
$$ 
We can see that this semigroup is a subsemigroup of the transformation semigroup that preserves equivalence, $T_E(X)$, and it is defined as the set of all functions in $T_E(X)$ that preserve the reverse direction of the equivalence relation $E$ on $X$. If $E$ is the universal relation on $X$, then $T_{E^*}(X)$ is equal to $T(X)$, which means that this semigroup is also a generalization of $T(X)$. Additionally, $T_{E^*}(X)$ is a semigroup consisting of continuous self-maps of the topological space $X$ where the $E$-classes form a basis. It is commonly referred to as a semigroup of continuous functions (see \cite{Magill} for details). In \cite{Deng}, the authors investigated the regularity of elements and Green's relations in $T_{E^*}(X)$. In 2013, L. Sun and J. Sun \cite{Sun2} gave a characterization of the natural partial order and determined the compatibility property in $T_{E^*}(X)$. They also described the maximal and minimal elements and studied the existence of the greatest lower bound of two elements in this semigroup.

In a recent study, K. Sangkhanan \cite{Sangkhanan} investigated the regular part, denoted $\reg(T)$, of the transformation semigroup $T_{E^*}(X)$ and showed that it is the largest regular subsemigroup of $T_{E^*}(X)$. He also described Green's relations and ideals of $\reg(T)$. If the set $X$ is partitioned by the equivalence relation $E$ into subsets $A_i$ for all $i$ in the index set $I$, the author defined the subsemigroup $Q(\mathbf{2})$ of $\reg(T)$ as follows:
$$
Q(\mathbf{2})=\{\alpha\in \reg(T):|A_i\alpha|<2\ \text{for all}\ i\in I\},
$$
or, equivalently,
$$
Q(\mathbf{2})=\{\alpha\in T_{E^*}(X):|A_i\alpha|=1\ \text{and}\ A_i\cap X\alpha\neq\emptyset\ \text{for all}\ i\in I\}.
$$
He also proved that for each $\alpha\in Q(\mathbf{2})$, $|X\alpha\cap A|=1$ for all $A\in X/E$, meaning that $X\alpha$ is a cross section of the partition $X/E$ induced by the equivalence relation $E$, i.e., each $E$-class contains exactly one element of $X\alpha$. He then showed that $Q(\mathbf{2})$ is the (unique) minimal ideal of $\reg(T)$, which is referred to as the \textit{kernel} of $\reg(T)$ (see \cite{Clifford1} for details). Finally, the author demonstrated that the kernel $Q(\mathbf{2})$ of $\reg(T)$ is a right group and can be expressed as a union of symmetric groups, and that every right group can be embedded in the kernel $Q(\mathbf{2})$.

To improve clarity, we will replace the use of $Q(\mathbf{2})$ with $Q_{E^*}(X)$, where $E$ represents the same equivalence relation as in $T_{E^*}(X)$.

%%%%%%%%%%%%%%%%%%%%%%%%%%%%%Basic Properties%%%%%%%%%%%%%%%%%%%%%%%%%%%%%%%%%%%%%%%%%%%%%%%%%%%%%%%%%%%%%%%%%%%%%%%%%%%%%%%%%%%%%%%%%%%

\section{Basic Properties}
Let $X$ be a nonempty set and $E$ an equivalence relation on $X$. The family of all equivalence classes, denoted $X/E$, is a partition of $X$. For each $\alpha\in T(X)$ and $A\subseteq X$, let $A\alpha=\{a\alpha:a\in A\}$. Evidently, by this notation, $X\alpha$ means the range or image of $\alpha$. Let $S$ be a subsemigroup of $T(X)$. The partition of a member $\alpha$ in $S$, denoted $\pi(\alpha)$, is the family of all inverse image of elements in the range of $\alpha$, that is,
$$
\pi(\alpha)=\{x\alpha^{-1}:x\in X\alpha\}.
$$
It is easy to see that $\pi(\alpha)$ is a partition of $X$ induced by $\alpha$ and $\pi(\alpha)=X/\ker(\alpha)$ where $\ker(\alpha)=\{(x,y)\in X\times X:x\alpha=y\alpha\}$ is an equivalence relation on $X$. In addition, due to \cite{Deng}, define a mapping $\alpha_*$ from $\pi(\alpha)$ onto $X\alpha$ by
\begin{center}{$(x\alpha^{-1})\alpha_*=x$ for each $x\in X\alpha.$}\end{center}

We recall that the relations $\mathscr{L},$ $\mathscr{R},$ $\mathscr{H},$ $\mathscr{D}$ and $\mathscr{J}$ are {\it Green's relations} on a semigroup $S$. For each $a\in S$, we denoted $\mathscr{L}$-class, $\mathscr{R}$-class, $\mathscr{H}$-class, $\mathscr{D}$-class, and $\mathscr{J}$-class containing $a$ by $L_a,R_a,H_a,D_a$ and $J_a$, respectively. 

A semigroup $S$ is called a \textit{right simple semigroup} if it does not contain any proper right ideals, and it is called a \textit{right group} if it is both right simple and left cancellative. This is equivalent to saying that a semigroup $S$ is a right group if and only if, for any elements $a$ and $b$ in $S$, there is only one element $x$ in $S$ such that $ax=b$. Therefore, the $\mathscr{R}$-relation on a right group $S$ is trivial. Additional descriptions of right groups can be found in the following lemma, which combines statements from Exercises 2 and 4 in \S1.11 of \cite{Clifford1}.

\begin{lem}\label{111}
	Let $S$ be a semigroup. The following statements are equivalent.
	\begin{enumerate}[(1)]
		\item $S$ is a right group.
		\item $S$ is a union of disjoint groups such that the set of identity elements of the groups is a right zero subsemigroup of $S$.
		\item $S$ is regular and left cancellative.
	\end{enumerate}
\end{lem}

Refer to \cite[Exercise 6 for \S2.6]{Howie} and \cite[Exercise 3 for \S1.11]{Clifford1}, every right group $S$ can be written as a union of disjoint subgroups, each of which is isomorphic to one another. These subgroups are given by $Se$, where $e$ is an idempotent element of $S$ and serves as the group identity for $Se$. If $e$ and $f$ are distinct idempotent elements of $S$, then the map $x\mapsto xf$ is an isomorphism between the subgroups $Se$ and $Sf$. Additionally, the $\mathcal{H}$-class $H_e$ and the subgroup $Se$ are equal for all idempotent elements $e$ of $S$. This allows us to express $S$ as the (disjoint) union of all of its subgroups:
$$
S=\bigcup_{e\in E(S)}Se=\bigcup_{e\in E(S)}H_e.
$$
It should be observed that the number of $\mathcal{H}$-classes and idempotents in $E(S)$ are identical.

In particular, for the right group $Q_{E^*}(X)$, we have the following results which appeared in \cite{Sangkhanan}.

\begin{cor}[{\cite[Corollary 3.8]{Sangkhanan}}]
	For each $\alpha\in Q_{E^*}(X)$, $H_\alpha=\{\beta\in Q_{E^*}(X):X\alpha= X\beta\}$ forms a subgroup of  $Q_{E^*}(X)$.
\end{cor}

\begin{thm}[{\cite[Theorem 3.9]{Sangkhanan}}]
	$Q_{E^*}(X)$ is a union of symmetric groups.
\end{thm}

By the proof of [{\cite[Theorem 3.9]{Sangkhanan}}], we obtain the following remark.

\begin{rem}
	Let $\alpha\in Q_{E^*}(X)$. Then $H_\alpha$ is isomorphic to the symmetric group on $X\alpha$.
\end{rem}

As a result of \cite{Clifford1}, we have a characterization of Green's $\mathcal{R}$ relations on the full transformation semigroup $T(X)$, where $X$ is an arbitrary set as follows.

\begin{thm}\label{tt}
	For each $\alpha,\beta\in T(X)$. Then $(\alpha,\beta)\in\mathscr{R}$ if and only if $\pi(\alpha)=\pi(\beta)$.
\end{thm}

Moreover, we can prove the following theorem, which extends Exercise 8 from $\S$2.2 in \cite{Clifford1}.

\begin{thm}\label{m1}
	Let $X$ be a nonempty set and $S$ a subsemigroup of $T(X)$. The following statements are equivalent.
\begin{enumerate}[(1)]
	\item $S$ is a right group.
	\item All members of $S$ have the same partition $\mathscr{P}$ of $X$ such that $X\alpha$ is a cross section of $\mathscr{P}$ for all $\alpha\in S$.
	\item $S$ is a regular subsemigroup of $Q_{E^*}(X)$ for some equivalence relation $E$ on $X$.
\end{enumerate}
\end{thm}
\begin{proof}
	(1) $\Rightarrow$ (2) Let $S$ be a right group. Since $\mathscr{R}$-relation on $S$ is trivial, we obtain that all members of $S$ have the same partition by Theorem \ref{tt}. For convenience, we denote such a partition of $X$ by $\mathscr{P}=\{P_i:i\in I\}$ where $I$ is an indexed set. Clearly, $|P_i\alpha|=1$ for all $i\in I$. It remains to show that $X\alpha$ is a cross section of $\mathscr{P}$ for all $\alpha\in S$.

	Let $\alpha\in S$. First, we will show that $|P_i\cap X\alpha|\leq1$ for all $i\in I$. Suppose to a contrary that $|P_0\cap X\alpha|>1$ for some $P_0\in\mathscr{P}$. Then there exist two distinct elements $a,b\in P_0\cap X\alpha$. We note that $a\alpha^{-1},b\alpha^{-1}\in\mathscr{P}$ such that $a\alpha^{-1}\neq b\alpha^{-1}$. Let $x\in a\alpha^{-1}$ and $y\in b\alpha^{-1}$. Since $|P_0\alpha|=1$, we obtain $a\alpha=b\alpha$ which implies that 
	$$
	x\alpha^2=(x\alpha)\alpha=a\alpha=b\alpha=(y\alpha)\alpha=y\alpha^2.
	$$
	Since $\alpha^2\in S$, we get $x$ and $y$ are in the same class in $\mathscr{P}$ which is a contradiction. Thus $|P_i\cap X\alpha|\leq1$ for all $i\in I$. Now, suppose that $P_1\cap X\alpha=\emptyset$ for some $P_1\in\mathscr{P}$. We note that $\alpha=\alpha\beta\alpha$ for some $\beta\in S$ since $S$ is a right group which is regular. Assume that 
	\begin{center}
		$P_1\beta_*\in P_2$ and $P_2\alpha_*\in P_3$
	\end{center}
	for some $P_2,P_3\in\mathscr{P}$. We can see that $P_3\not=P_1$ since $X\alpha\cap P_1=\emptyset$. From $\alpha=\alpha\beta\alpha$, we obtain that
	$$
	(P_3\beta_*)\alpha=(P_2\alpha_*)\beta\alpha=P_2\alpha_*\in P_3
	$$
	and
	$$
	(P_1\beta_*)\alpha=P_2\alpha_*\in P_3.
	$$
	Hence $(P_3\beta_*)\alpha=(P_1\beta_*)\alpha$ since $|P_3\cap X\alpha|\leq1$. By $\beta\alpha\in S$, it is concluded that $P_3$ and $P_1$ are the same class in $\mathscr{P}$ which is a contradiction. Therefore, $|P_i\cap X\alpha|=1$ for all $i\in I$ from which it follows that $X\alpha$ is a cross section of $\mathscr{P}$ for all $\alpha\in S$.
	
	(2) $\Rightarrow$ (3). Suppose that (2) holds. Let $E$ be an equivalence relation on $X$ induced by $\mathscr{P}=\{P_i:i\in I\}$. Then $S\subseteq Q_{E^*}(X)$ by the definition of $Q_{E^*}(X)$. We can see that $S$ is closed since $S$ is a subsemigroup of $T(X)$. Moreover, $S$ is regular by Theorem 3.1 of \cite{Deng}.

	(3) $\Rightarrow$ (1). Suppose that (3) holds. Then $S$ is regular. We can see that $S$ is left cancellative since $Q_{E^*}(X)$ is a right group. Hence $S$ is also a right group.
\end{proof}

According to \cite[Theorem 1.27]{Clifford1}, a right group $S$ is isomorphic to the direct product $G\times E$ of a group $G$ and a right zero semigroup $E$ where $E$ is the set $E(S)$ of all idempotents in $S$ and $G$ is the group $Se$ $(=H_e)$ where $e\in E(S)$.

To prove the condition for $Q_{E^*}(X)$ to be a group, we first state a result appeared in \cite{Harden}.

\begin{thm}[{\cite[Theorem 2.8]{Harden}}]\label{thm: direct product of two groups}
	The semigroup $S\times T$ is a group if and only if the semigroups $S$ and $T$ are both groups.
\end{thm}

From the above theorem, we have the following result.

\begin{prop}\label{prop: right group is a group}
	A right group $S$ is a group if and only if $E(S)$ is a trivial semigroup.
\end{prop}
\begin{proof}
	Let $S$ be a right group. Assume that $E(S)$ is a trivial semigroup which is also a group. As we mentioned above, $S$ can be written as a direct product of a group $G$ and $E(S)$. By Theorem \ref{thm: direct product of two groups}, we obtain $S\cong G\times E(S)$ is a group. The converse is clear since every group has a unique idempotent.
\end{proof}

Let $E(Q_{E^*}(X))$ be the set of all idempotents in $Q_{E^*}(X)$. The following lemma is a characterization of idempotents in $Q_{E^*}(X)$.

\begin{lem}[{\cite[Lemma 4.1]{Sangkhanan}}]\label{00}
	$\alpha\in Q_{E^*}(X)$ is an idempotent if and only if  $A\alpha\subseteq A$ for all $A\in X/E$.
\end{lem}

We recall that the relation $I_X=\{(x,x):x\in X\}$ on a set $X$ is called the {\it identity relation}. The next theorem shows that $Q_{E^*}(X)$ is almost never a group.

\begin{thm}
	$Q_{E^*}(X)$ is a group if and only if $E=I_X$.
\end{thm}
\begin{proof}
	Let $X/E=\{A_i:i\in I\}$. Suppose that $Q_{E^*}(X)$ is a group. Then by Proposition \ref{prop: right group is a group}, $E(Q_{E^*}(X))$ is trivial. Let us assume to the contrary that $E\not=I_X$. Then there exists an equivalence class $A_0$ in $X/E$ which has more than one element. Let $x$ and $y$ be distinct elements in $A_0$. For each $j\in I\setminus\{0\}$, fix $a_j\in A_j$ and define
	$$
	\alpha=\begin{pmatrix}A_0 & A_j\\x & a_j\end{pmatrix}\ \text{and}\ \beta=\begin{pmatrix}A_0 & A_j\\y & a_j\end{pmatrix}.
	$$
	It is clearly that $\alpha,\beta\in Q_{E^*}(X)$. We can see that $\alpha$ and $\beta$ are idempotents by Lemma \ref{00}. It leads to a contradiction since $\alpha\not=\beta$. 
	
	Conversely, let $E=I_X$. Then $Q_{E^*}(X)$ is a symmetric group on the set $X$ which implies that $Q_{E^*}(X)$ is a group.
\end{proof}

%%%%%%%%%%%%%%%%%%%%%%%%%%%%%%%%Isomorphism Conditions%%%%%%%%%%%%%%%%%%%%%%%%%%%%%%%%%%%%%%%%%%%%%%%%%%%%%%%%%%%%%%%%%%%%%%%%%%%%%%%%%%%%%%%%

\section{Isomorphism Conditions}

The goal of this section is to determine a necessary and sufficient condition for the existence of an isomorphism between $Q_{E^*}(X)$ and $Q_{F^*}(Y)$, where $E$ and $F$ are equivalence relations on $X$ and $Y$, respectively.

To establish a criterion for isomorphism, we initially present a useful result in the following manner.

\begin{thm}\label{thm: right group iso}
	Let $S$ and $T$ be right groups which can be written as direct products of groups and right zero semigroups $G_1\times E_1$ and $G_2\times E_2$, respectively. Then $S\cong T$ if and only if $G_1\cong G_2$ and $|E_1|=|E_2|$.
\end{thm}
\begin{proof}
	Assume that $S\cong T$ via an isomorphism $\phi:S\to T$. As we mentioned above, $G_1\cong Se$ for some idempotent $e\in E(S)$. It is straightforward to verify that $Se\cong T(e\phi)$. Hence $G_1\cong Se\cong T(e\phi)=Tf\cong G_2$ for some $f=e\phi\in E(T)$. Moreover, $E_1\cong E(S)\cong E(T)\cong E_2$ which implies $|E_1|=|E_2|$ since any two right zero semigroups of the same cardinality are isomorphic. The converse is clear.
\end{proof}

We represent the symmetric group on a nonempty set $X$ as $S_X$. For an indexed collection of sets $\{A_i\}_{i\in I}$, we denote their product as $\prod\limits_{i\in I}A_i$, and an element in $\prod\limits_{i\in I}A_i$ with $i$-coordinate $a_i$ is designated as $(a_i)_{i\in I}$. The expression $\prod\limits_{i\in I}|A_i|$ denotes the product of the cardinal numbers of $A_i$ for all $i\in I$. It is a well-established fact that $\prod\limits_{i\in I}|A_i|=\left|\prod\limits_{i\in I}A_i\right|$.

By a combination of \cite[Exercise 10, pp. 40 and Exercise 8, pp. 151]{Dummit}, we obtain an isomorphism condition for two symmetric groups as follows.

\begin{lem}\label{32}
	Let $A$ and $B$ be any nonempty sets. Then
\begin{center}
	$S_A\cong S_B$ if and only if $|A|=|B|$.
\end{center}
\end{lem}

To characterize an isomorphism condition, the following lemma is needed.
\begin{lem}\label{01}
	Let $X/E=\{A_i:i\in I\}$. Then $|E(Q_{E^*}(X))|=\prod\limits_{i\in I}|A_i|$.
\end{lem}
\begin{proof}
	For each $\epsilon\in E(Q_{E^*}(X))$, we have $A_i\epsilon\subseteq A_i$ for all $i\in I$ by Lemma \ref{00}. So we can write
	$$
	\epsilon=\begin{pmatrix}A_i\\a_i\end{pmatrix}\begin{text}{~where~$a_i\in A_i$.}\end{text}
	$$
	Define a function $\varphi:E(Q_{E^*}(X))\rightarrow\prod\limits_{i\in I}A_i$ by
	$$
	\epsilon\varphi=\left(A_i\epsilon_*\right)_{i\in I}=(a_i)_{i\in I}\in\prod_{i\in I}A_i.
	$$
	To show that $\varphi$ is an injection, let $\epsilon_1=\begin{pmatrix}A_i\\a_i\end{pmatrix}$ and $\epsilon_2=\begin{pmatrix}A_i\\b_i\end{pmatrix}$ in $E(Q_{E^*}(X))$ such that $\epsilon_1\varphi=\epsilon_2\varphi$. Then $a_i=A_i{\epsilon_1}_*=A_i{\epsilon_2}_*=b_i$ for all $i\in I$. Hence $\epsilon_1=\epsilon_2$ which implies that $\varphi$ is an injection. Finally, let $(a_i)_{i\in I}\in\prod\limits_{i\in I}A_i$. Define $\epsilon\in E(Q_{E^*}(X))$ by $\epsilon=\begin{pmatrix}A_i\\a_i\end{pmatrix}$. We obtain $(a_i)_{i\in I}=\left(A_i\epsilon_*\right)_{i\in I}=\epsilon\varphi$. Hence $\varphi$ is a surjective map, implying that $\varphi$ is also a bijection. Therefore, $|E(Q_{E^*}(X))|=\left|\prod\limits_{i\in I}A_i\right|=\prod\limits_{i\in I}|A_i|$.
\end{proof}

As we mentioned before, a right group $S$ is isomorphic to the direct product $G\times E$ of a group $G$ and a right zero semigroup $E$ where $E$ is the set $E(S)$ of all idempotents in $S$ and $G$ is the group $Se$ $(=H_e)$ where $e\in E(S)$. By \cite[Theorem 3.8]{Sangkhanan}, the author proved that the subgroup $H_\alpha$ of $Q_{E^*}(X)$ is isomorphic to the symmetric group on the set $X\alpha$. Since $X\alpha$ is a cross section of the partition $X/E$, it is obvious that the symmetric group on the set $X\alpha$ is isomorphic to the symmetric group on $X/E$. To sum it up, we state the useful proposition as follows.

\begin{prop}\label{prop: Q iso direct product}
	$Q_{E^*}(X)$ is isomorphic to the direct product of $S_{X/E}$ and $E(Q_{E^*}(X))$.
\end{prop}

Now, we obtain an isomorphism condition between two right groups $Q_{E^*}(X)$ and $Q_{F^*}(Y)$ as follows.

\begin{thm}
	Let $E$ and $F$ be equivalence relations on nonempty sets $X$ and $Y$, respectively. Then $Q_{E^*}(X)\cong Q_{F^*}(Y)$ if and only if $|X/E|=|Y/F|$ and $\prod\limits_{A\in X/E}|A|=\prod\limits_{B\in Y/F}|B|$.
\end{thm}
\begin{proof}
	 Assume that $Q_{E^*}(X)\cong Q_{F^*}(Y)$. By Theorem \ref{thm: right group iso} and Proposition \ref{prop: Q iso direct product}, we obtain $S_{X/E}\cong S_{Y/F}$ which implies by Lemma \ref{32} that $|X/E|=|Y/F|$. Moreover, we get $|E(Q_{E^*}(X))|=|E(Q_{F^*}(Y))|$ which follows by Lemma \ref{01} that $\prod\limits_{A\in X/E}|A|=\prod\limits_{B\in Y/F}|B|$.

	Conversely, suppose that $|X/E|=|Y/F|$ and $\prod\limits_{A\in X/E}|A|=\prod\limits_{B\in Y/F}|B|$. By Lemma \ref{32}, we have $S_{X/E}\cong S_{Y/F}$. Moreover, $|E(Q_{E^*}(X))|=|E(Q_{F^*}(Y))|$ by Lemma \ref{01}. Since any two right zero semigroups of the same cardinality are isomorphic, we obtain $E(Q_{E^*}(X))\cong E(Q_{F^*}(Y))$. Therefore, again by Theorem \ref{thm: right group iso} and Proposition \ref{prop: Q iso direct product}, $Q_{E^*}(X)\cong Q_{F^*}(Y)$.
\end{proof}

We conclude this section by calculating the cardinality of the set $Q_{E^*}(X)$. Let $X$ be a finite set and $E$ an equivalence relation on $X$ such that $X/E=\{A_1,A_2,\ldots,A_n\}$ and $|A_1||A_2|\cdots|A_n|=m$. By Proposition \ref{prop: Q iso direct product} and Lemma \ref{01}, we have
$$
|Q_{E^*}(X)|=|S_{X/E}||E(Q_{E^*}(X))|=n!\cdot m.
$$

%%%%%%%%%%%%%%%%%%%%%%%%%%Ranks%%%%%%%%%%%%%%%%%%%%%%%%%%%%%%%%%%%%%%%%%%%%%%%%%%%%%%%%%%%%%%%%%%%%%%%%%%%%%%%%%%%%%%%%%%%%%%

\section{Ranks}
	
In this section, we will find a generating set and compute the rank of $Q_{E^*}(X)$. Recall that the rank of a semigroup $S$, represented as $\rank(S)$, can indicate the minimum number of elements needed to generate $S$, that is
$$
\rank(S)=\min\{|X|:X\subseteq S,\langle X\rangle =S\}.
$$
To compute the rank of $Q_{E^*}(X)$, we consider a generating set of any right groups. The following lemmas will be necessary in order to prove the main theorem of this section.

\begin{lem}\label{ge} Let $G$ be a generating set of a right group $S$. Then $G\cap H_e$ is nonempty for all idempotents $e$ in $S$.
\end{lem}
\begin{proof}
	 Let $e\in E(S)$. Then there are $g_1,g_2,\ldots,g_m\in G$ such that $g_1g_2\cdots g_m=e$. Let $g_m\in H_f$ for some idempotent $f$ in $S$. Then $e=g_1g_2\cdots g_mf=ef=f$ since $f$ is the identity in $H_f$ and $E(S)$ is right zero. Hence $g_m\in G\cap H_e\neq\emptyset$.
\end{proof}

By the above lemma, any generating set $G$ of a right group $S$ contains at least one element in each $\mathcal{H}$-class. Hence $\rank(S)\geq|E(S)|$.

At this point, we have the capability to acquire a generating set for any right group.

\begin{thm}\label{gen}
	Let $S$ be a right group and $e\in E(S)$. If $G$ is a generating set of the group $H_e$, then $G\cup E(S)$ is a generating set of $S$.
\end{thm}
\begin{proof}
	Assume that $G$ is a generating set of the group $H_e$. Let $x\in S$. Then $x\in H_f$ for some idempotent $f$ in $S$. We have $xe\in Se=H_e$ and $xe=g_1g_2\cdots g_m$ for some $g_1,g_2,\ldots,g_m\in G$. Hence $x=xf=xef=g_1g_2\cdots g_mf$ which implies that $x\in\langle G\cup E(S)\rangle$.
\end{proof}

\begin{thm}\label{thm: rank of right group}
	Let $S$ be a right group and $e\in E(S)$. If $G$ is a minimal generating set of the group $H_e$, then $\rank(S)=\max\{|G|,|E(S)|\}$.
\end{thm}
\begin{proof}
	Suppose that $G$ is a minimal generating set of the group $H_e$. If $|G|\leq|E(S)|$, then there is an injection $\phi:G\to E(S)$. Refer to \cite[Theorem 1.27]{Clifford1}, we have $S\cong H_e\times E(S)$. Let $H=\{(g,g\phi):g\in G\}$ and $K=\{(e,f):f\in E(S)\setminus\im\phi\}$. We assert that $H\cup K$ is a generating set of $H_e\times E(S)$. For, let $(x,h)\in H_e\times E(S)$. If $h=g\phi$ for some $g\in G$, then $xg^{-1}=g_1g_2\cdots g_m$ for some $g_1,g_2,\ldots,g_m\in G$. Hence
	$$
	(x,h)=(xg^{-1}g,h)=(g_1g_2\cdots g_mg,g\phi).
	$$
	Since $E(S)$ is right zero, we obtain
	$$
	(g_1g_2\cdots g_mg,g\phi)=(g_1,g_1\phi)(g_2,g_2\phi)\cdots(g_m,g_m\phi)(g,g\phi)
	$$
	which implies that $(x,h)\in\langle H\rangle\subseteq\langle H\cup K\rangle$. If $h\in E(S)\setminus\im\phi$, then $x=g_1g_2\cdots g_m$ for some $g_1,g_2,\ldots,g_m\in G$ and so
	$$
	(x,h)=(xe,h)=(g_1g_2\cdots g_me,h)=(g_1,g_1\phi)(g_2,g_2\phi)\cdots(g_m,g_m\phi)(e,h).
	$$
	Thus $(x,h)\in\langle H\cup K\rangle$. We conclude that 
	$$
	\rank(S)=\rank(H_e\times E(S))\leq|H\cup K|=|H|+|K|=|G|+(|E(S)|-|G|)=|E(S)|.
	$$
	As mentioned before, we have $\rank(S)\geq|E(S)|$. Therefore,
	$$
	\rank(S)=|E(S)|=\max\{|G|,|E(S)|\}.
	$$

	On the other hand, assume that $|G|>|E(S)|$. Then there is a surjection $\phi:G\to E(S)$. Let $H=\{(g,g\phi):g\in G\}$. By the same argument as above, we can show that $H$ is a generating set of $S$ (up to isomorphism). Hence $\rank(S)\leq|H|=|G|$. Since $G$ is a minimal generating set of the group $H_e$, we obtain $\rank(S)=|G|=\max\{|G|,|E(S)|\}$.
\end{proof}

Let $X/E=\{A_i:i\in I\}$. We have $|E(Q_{E^*}(X))|=\prod\limits_{i\in I}|A_i|$. Refer to Proposition \ref{prop: Q iso direct product}, $Q_{E^*}(X)$ is isomorphic to the direct product of $S_{X/E}$ and $E(Q_{E^*}(X))$. Moreover, it is well-known that the symmetric group on a set $Y$ has rank $2$ when $|Y|\geq2$. By using Theorem \ref{thm: rank of right group}, we obtain the rank of $Q_{E^*}(X)$ as follows.

\begin{cor}\label{rank}
	Let $E$ be a nontrivial equivalence relation on a nonempty set $X$. Let $X/E=\{A_i:i\in I\}$ be such that $\prod\limits_{i\in I}|A_i|=m$. Then $\rank(Q_{E^*}(X))=\max\{2,m\}$.
\end{cor}

%%%%%%%%%%%%%%%%%%%%%%%%%%Maximal Subsemigroups%%%%%%%%%%%%%%%%%%%%%%%%%%%%%%%%%%%%%%%%%%%%%%%%%%%%%%%%%%%%%%%%%%%%%%%%%%%%%%%%%%%%%%%%%%%%%%

\section{Maximal Subsemigroups}

Let's revisit the concept of a maximal subsemigroup of a semigroup $S$. A maximal subsemigroup of $S$ refers to a subset of $S$ that is both a proper subsemigroup (i.e., not equal to $S$) and not contained within any other proper subsemigroup of $S$. Analogously, a maximal proper subgroup of a group $G$ can be defined as a subgroup that cannot be contained within any other proper subgroup of $G$. It is well-known that, when $G$ is a finite group, every subsemigroup of $G$ becomes a subgroup.

In this section, we will provide a description and enumeration of the maximal subsemigroups present in any right group which can be written as the direct product of a finite group and a right zero semigroup, as well as in $Q_{E^*}(X)$, where $X$ is a finite set. We begin this section by stating the following lemma.

\begin{lem}\label{lem: right group when G finite}
	Let $S$ be a right group which can be written as the direct product of a finite group $G$ and a right zero semigroup $E$. Then every subsemigroup of $S$ is also a right group.
\end{lem}
\begin{proof}
	Let $T$ be a subsemigroup of $S$. Then $T$ can be written as the direct product of a subsemigroup $H$ of $G$ and a subsemigroup $F$ of $E$. It follows that $H$ is a subgroup of $G$ since $G$ is finite. Moreover, $F$ is also a right zero semigroup. Therefore, $T$ can be written as the direct product of the group $H$ and the right zero semigroup $F$, and thus it is also a right group.
\end{proof}

By the above lemma, we obtain the following results immediately.

\begin{prop}\label{prop: Every subsemigroup of a finite right group is a right group}
	Every subsemigroup of a finite right group is also a right group.
\end{prop}

\begin{lem}\label{lem: subsemigroup equal}
	Let $S$ be a right group which can be written as the direct product of a finite group $G$ and a right zero semigroup $E$. Let $T$ be a subsemigroup of $S$ and $e\in E(S)$. If $Te=Se$ and $E(T)=E(S)$, then $T=S$. 
\end{lem}
\begin{proof}
	We note by Lemma \ref{lem: right group when G finite} that $T$ is a right group. Assume that $Te=Se$ and $E(T)=E(S)$. Then $Tf=Tef=Sef=Sf$ for all $f\in E(S)$. Hence
	$$
	T=\bigcup_{f\in E(T)}Tf=\bigcup_{f\in E(S)}Tf=\bigcup_{f\in E(S)}Sf=S.
	$$
\end{proof}

To characterize a maximal subsemigroup of a right group, we need the following lemma which appeared in \cite{Wilson}.

\begin{lem}[{\cite[Lemma 4.3]{Wilson}}]\label{lem: maximal subsemigroup lack one element}
	Let $S$ be a semigroup and let $M$ be a subsemigroup of $S$ such that $|S\setminus M|=1$. Then $M$ is a maximal subsemigroup of $S$.
\end{lem}

By Lemma \ref{lem: maximal subsemigroup lack one element} and the dual statement of Example 4.4 in \cite{Wilson}, we have the following proposition.

\begin{prop}\label{prop: maximal subsemigroup of right zero}
	Let $S$ be a right zero semigroup. Then $M$ is a maximal subsemigroup of $S$ if and only if $M=S\setminus\{x\}$ for some $x\in S$.
\end{prop}

By the above proposition, we conclude that every nontrivial right zero semigroup has a maximal subsemigroup. In addition, we note that if a finite right zero semigroup $S$ has $m$ elements, then the number of its maximal subsemigroups is also $m$.

Now, we provide a characterization of maximal subsemigroups of any right group which can be written as the direct product of a finite group and a right zero semigroup.

\begin{thm}\label{thm: maximal subsemigroup of finite right group}
	Let $S$ be a right group which can be written as the direct product of a finite group $G$ and a nontrivial right zero semigroup $E(S)$, and let $T$ be a subsemigroup of $S$. Then $T$ is a maximal subsemigroup of $S$ if and only if $T$ can be written as a direct product $H\times E(S)$ or $G\times F$ where $H$ is a maximal subgroup of $G$ and $F$ is a maximal subsemigroup of $E(S)$.
\end{thm}
\begin{proof}
	Assume that $T$ is a maximal subsemigroup of $S$. Then, by Lemma \ref{lem: right group when G finite}, $T$ is a right group which can be written as a direct product $Te\times E(T)$ where $e$ is an idempotent in $T$. We have $Te$ is a subgroup of $Se$ and $E(T)$ is a right zero subsemigroup of $E(S)$.
	
	If $Te$ is a proper subgroup of $Se$, then there is a maximal subgroup $H$ of $Se$ such that $Te\subseteq H\subsetneq Se$. Clearly, $H\times E(S)$ is a subsemigroup of $Se\times E(S)$ and $Te\times E(T)\subseteq H\times E(S)\subsetneq Se\times E(S)$. Since $Te\times E(T)$ is maximal, we obtain $Te\times E(T)=H\times E(S)$ and so $Te=H$ and $E(T)=E(S)$. 
	
	If $E(T)$ is a proper subsemigroup of $E(S)$, then there is a maximal subsemigroup $F$ of $E(S)$ such that $E(T)\subseteq F\subsetneq E(S)$. We have $Te\times E(T)\subseteq Se\times F\subsetneq Se\times E(S)$. Since $Te\times E(T)$ is maximal, we obtain $Te\times E(T)=Se\times F$ and so $Te=Se$ and $E(T)=F$. 
	
	Conversely, suppose that $T$ can be written as a direct product $Te\times E(S)$ where $Te$ ($e\in E(T)$) is a maximal subgroup of $Se$. To show that $T$ is maximal, let $U$ be a subsemigroup of $S$ such that $T\subseteq U\subseteq S$. Again by Lemma \ref{lem: right group when G finite}, $U$ is a right group which can be written as the direct product $Ue\times E(U)$ where $Ue$ is a subgroup of $Se$ and $E(U)$ is a right zero subsemigroup of $E(S)$. Clearly,
	$$
	Te\times E(S)\subseteq Ue\times E(U)\subseteq Se\times E(S).
	$$ 
	Hence $E(U)=E(S)$. By maximality of $Te$, we obtain $Te=Ue$ or $Ue=Se$ which implies by Lemma \ref{lem: subsemigroup equal} that $T=U$ or $U=S$. It is concluded that $T$ is maximal.

	Finally, assume that $T$ can be written as a direct product $Se\times E(T)$ where $Te=Se$ and $E(T)$ is a maximal subsemigroup of $E(S)$. To show that $T$ is maximal, let $U$ be a subsemigroup of $S$ such that $T\subseteq U\subseteq S$. By the same argument as above, we can write
	$$
	Se\times E(T)\subseteq Ue\times E(U)\subseteq Se\times E(S).
	$$ 
	Hence $Te=Se=Ue$. By maximality of $E(T)$, we obtain $E(T)=E(U)$ or $E(U)=E(S)$. Again by Lemma \ref{lem: subsemigroup equal}, $T=U$ or $U=S$ and so $T$ is maximal.
\end{proof}

As a direct consequence of Theorem \ref{thm: maximal subsemigroup of finite right group} and Proposition \ref{prop: maximal subsemigroup of right zero}, we obtain the following corollary.

\begin{cor}\label{cor: maximal subsemigroup of right group}
	Let $S$ be a right group which can be written as the direct product of a finite group $G$ and a nontrivial right zero semigroup $E(S)$, and let $T$ be a subsemigroup of $S$. Then $T$ is a maximal subsemigroup of $S$ if and only if $T$ can be written as the direct product $H\times E(S)$ or $G\times F$ where $H$ is a maximal subgroup of $G$ and $F=E(S)\setminus\{e\}$ for some $e\in E(S)$.
\end{cor}

Let $S$ be a finite right group which can be written as the direct product of a group $G$ and a nontrivial right zero semigroup $E(S)$, where the number of maximal subgroup of $G$ is $n$ and $|E(S)|=m>1$. We also note by the above corollary that the number of maximal subsemigroup of $S$ is $n+m$. Furthermore, let $X$ be a finite set and $E$ an equivalence relation on $X$ which is not the identity relation. If $X/E=\{A_1,A_2,\ldots,A_n\}$ and $|A_1||A_2|\cdots|A_n|=m$, then the number of maximal subsemigroups of $Q_{E^*}(X)$ is $s_n+m$ where $s_n$ is the number of maximal subgroups of the symmetric group of order $n$ (see \cite[A290138]{Mitchell} for details).

\section{Examples}

In this section, we show an example of $Q_{E^*}(X)$ when $X$ is finite and then we will find its rank, a minimal generating set and all maximal subsemigroups which corresponds to the previous sections.

Let $X/E=\{A_1,A_2,\ldots,A_n\}$. For convenience, we denote an element
$$
\alpha=\begin{pmatrix}A_1 & A_2 & \cdots & A_n\\a_1 & a_2 & \cdots & a_n\end{pmatrix}
$$
in $Q_{E^*}(X)$ by $\alpha=(a_1,a_2,\ldots,a_n)$.

Let $X=\{1,2,3,4,5,6\}$ and let $E$ be an equivalence relation on $X$ such that $X/E=\{A_1,A_2,A_3\}$ where $A_1=\{1,2,3\}$, $A_2=\{4,5\}$ and $A_3=\{6\}$. We have
$$
Q_{E^*}(X)=\{\alpha_1,\alpha_2,\ldots,\alpha_{36}\}
$$
where  
$$
\begin{matrix}
\alpha_{1}=(1, 4, 6), & \alpha_{2}=(2, 4, 6), & \alpha_{3}=(3, 4, 6), & \alpha_{4}=(1, 5, 6), & \alpha_{5}=(2, 5, 6), \\ 
\alpha_{6}=(3, 5, 6), & \alpha_{7}=(4,1,6), & \alpha_{8}=(4,2,6), & \alpha_{9}=(4,3,6), & \alpha_{10}=(5,1, 6), \\ 
\alpha_{11}=(5,2, 6), & \alpha_{12}=(5,3, 6), & \alpha_{13}=(4,6,1), & \alpha_{14}=(4,6,2), & \alpha_{15}=(4,6,3), \\ 
\alpha_{16}=(5,6,1), & \alpha_{17}=(5,6,2), & \alpha_{18}=(5,6,3), & \alpha_{19}=(6,1,4), & \alpha_{20}=(6,2,4), \\ 
\alpha_{21}=(6,3,4), & \alpha_{22}=(6,1,5), & \alpha_{23}=(6,2,5), & \alpha_{24}=(6,3,5), & \alpha_{25}=(1, 6,4), \\ 
\alpha_{26}=(2,6,4), & \alpha_{27}=(3,6,4), & \alpha_{28}=(1,6,5), & \alpha_{29}=(2,6,5), & \alpha_{30}=(3,6,5), \\ 
\alpha_{31}=(6,4,1), & \alpha_{32}=(6,4,2), & \alpha_{33}=(6,4,3), & \alpha_{34}=(6,5,1), & \alpha_{35}=(6,5,2), \\ 
\alpha_{36}=(6,5,3).
\end{matrix}$$
Moreover, from Lemma \ref{00}, we obtain 
$$
E(Q_{E^*}(X))=\{\alpha_1,\alpha_2,\alpha_3,\alpha_4,\alpha_5,\alpha_6\}.
$$

To find a minimal generating set of $Q_{E^*}(X)$, we choose an idempotent $e=\alpha_1=(1,4,6)$. Then we obtain that the $\mathcal{H}$-class $H_e$ of $Q_{E^*}(X)$ is 
$$
\{(1,4,6),(4,1,6),(4,6,1),(6,1,4),(1,6,4),(6,4,1)\}=\{\alpha_1,\alpha_7,\alpha_{13},\alpha_{19},\alpha_{25},\alpha_{31}\}.
$$
Since $H_e$ is a symmetric group, $G=\{(1,4,6),(4,1,6)\}=\{\alpha_1,\alpha_7\}$ is a minimal generating set of $H_e$. Hence, by Theorem \ref{gen}, we have 
$$
G\cup\{\alpha_1,\alpha_2,\alpha_3,\alpha_4,\alpha_5,\alpha_6\}=\{\alpha_1,\alpha_2,\alpha_3,\alpha_4,\alpha_5,\alpha_6,\alpha_7\}
$$
is a generating set of $Q_{E^*}(X)$. Moreover, we can see that $\alpha_7^2=(4,1,6)^2=(1,4,6)=\alpha_1$ implies $Q_{E^*}(X)=\langle\alpha_2,\alpha_3,\alpha_4,\alpha_5,\alpha_6,\alpha_7\rangle$. Since $|\{\alpha_2,\alpha_3,\alpha_4,\alpha_5,\alpha_6,\alpha_7\}|=6=\max\{2,m\}$ where $m=|A_1||A_2||A_3|=3\cdot2\cdot1=6$. Therefore $\{\alpha_2,\alpha_3,\alpha_4,\alpha_5,\alpha_6,\alpha_7\}$ is a minimal generating set of $Q_{E^*}(X)$ by applying Corollary \ref{rank}.

Finally, we will find all maximal subsemigroups of $Q_{E^*}(X)$. First, consider 
$$
H_{\alpha_1}=\{\alpha_1,\alpha_7,\alpha_{13},\alpha_{19},\alpha_{25},\alpha_{31}\}.
$$ 
We have $H_{\alpha_1}$ is isomorphic to the symmetric group on $X\alpha_1$. Moreover, since $|X\alpha_1|=3$, we obtain by Lemma \ref{32} that the symmetric group on $X\alpha_1$ is isomorphic to the symmetric group of degree $3$, denoted by $S_3$. Note that the elements of symmetric group $S_3$ can be written in cycle notation as 
$$
\left\{\begin{pmatrix}&\end{pmatrix},\begin{pmatrix}1&2\end{pmatrix},\begin{pmatrix}1&3\end{pmatrix},\begin{pmatrix}2&3\end{pmatrix},\begin{pmatrix}1&2&3\end{pmatrix},\begin{pmatrix}1&3&2\end{pmatrix}\right\}
$$ 
where $\begin{pmatrix}&\end{pmatrix}$ is an identity permutation. It is easy to verify that $H_{\alpha_1}\cong S_3$ via the isomorphism $\psi:H_{\alpha_1}\rightarrow S_3$ defined by\\
$\alpha_{1}\psi=\big((1, 4, 6)\big)\psi=\begin{pmatrix}&\end{pmatrix},$\\
$\alpha_{7}\psi=\big((4,1,6)\big)\psi=\begin{pmatrix}1&2\end{pmatrix},$\\ 
$\alpha_{13}\psi=\big((4,6,1)\big)\psi=\begin{pmatrix}1&2&3\end{pmatrix},$\\
$\alpha_{19}\psi=\big((6,1,4)\big)\psi=\begin{pmatrix}1&3&2\end{pmatrix},$\\ 
$\alpha_{25}\psi=\big((1, 6,4)\big)\psi=\begin{pmatrix}2&3\end{pmatrix}$ and\\
$\alpha_{31}\psi=\big((6,4,1)\big)\psi=\begin{pmatrix}1&3\end{pmatrix}.$\\
In addition, it is well-known that all maximal subgroups of $S_3$ are
$$
\left\{\begin{pmatrix}&\end{pmatrix},\begin{pmatrix}1&2\end{pmatrix}\right\},\left\{\begin{pmatrix}&\end{pmatrix},\begin{pmatrix}2&3\end{pmatrix}\right\},\left\{\begin{pmatrix}&\end{pmatrix},\begin{pmatrix}1&3\end{pmatrix}\right\},\left\{\begin{pmatrix}&\end{pmatrix},\begin{pmatrix}1&2&3\end{pmatrix},\begin{pmatrix}1&3&2\end{pmatrix}\right\}.
$$
Then maximal subgroups of $H_{\alpha_1}$ are
$$
\begin{matrix}G_1=\{\alpha_1,\alpha_7\}, & G_2=\{\alpha_1,\alpha_{25}\}, & G_3=\{\alpha_1,\alpha_{31}\} &\begin{text}{ and }\end{text}& G_4=\{\alpha_1,\alpha_{13},\alpha_{19}\}\end{matrix}.
$$
From Proposition \ref{prop: maximal subsemigroup of right zero}, define a maximal subsemigroup $F_i$ of $E(Q_{E^*}(X))$ by
$$
F_i=E(Q_{E^*}(X))\setminus\{\alpha_i\}
$$
where $i\in\{1,2,3,\dots,6\}$. Refer to Exercise 2.6 (6) of \cite{Howie}, the map $\phi:H_{\alpha_1}\times E(Q_{E^*}(X))\rightarrow Q_{E^*}(X)$ defined by $(a,e)\phi=ae$ is an isomorphism and then we apply Corollary \ref{cor: maximal subsemigroup of right group} to identify and obtain all maximal subsemigroups of $Q_{E^*}(X)$ that are:
\begin{enumerate}[(1)]
	\item $T_1=\big(G_1\times E(Q_{E^*}(X))\big)\phi=\{\alpha_i:i\in\{1,2,3,\dots,12\}\}$,
	\item $T_2=\big(G_2\times E(Q_{E^*}(X))\big)\phi=\{\alpha_i:i\in\{1,2,3,\dots,6,25,26,27,\dots,30\}\}$,
	\item $T_3=\big(G_3\times E(Q_{E^*}(X))\big)\phi=\{\alpha_i:i\in\{1,2,3,\dots,6,31,32,33,\dots,36\}\}$,
	\item $T_4=\big(G_4\times E(Q_{E^*}(X))\big)\phi=\{\alpha_i:i\in\{1,2,3,\dots,6,13,14,15,\dots,24\}\}$,
	\item $T_5=\big(H_{\alpha_1}\times F_1\big)\phi=Q_{E^*}(X)\setminus\{\alpha_{1},\alpha_{7},\alpha_{13},\alpha_{19},\alpha_{25},\alpha_{31}\}$,
	\item $T_6=\big(H_{\alpha_1}\times F_2\big)\phi=Q_{E^*}(X)\setminus\{\alpha_{2},\alpha_{8},\alpha_{14},\alpha_{20},\alpha_{26},\alpha_{32}\}$,
	\item $T_7=\big(H_{\alpha_1}\times F_3\big)\phi=Q_{E^*}(X)\setminus\{\alpha_{3},\alpha_{9},\alpha_{15},\alpha_{21},\alpha_{27},\alpha_{33}\}$,
	\item $T_8=\big(H_{\alpha_1}\times F_4\big)\phi=Q_{E^*}(X)\setminus\{\alpha_{4},\alpha_{10},\alpha_{16},\alpha_{22},\alpha_{28},\alpha_{34}\}$,
	\item $T_9=\big(H_{\alpha_1}\times F_5\big)\phi=Q_{E^*}(X)\setminus\{\alpha_{5},\alpha_{11},\alpha_{17},\alpha_{23},\alpha_{29},\alpha_{35}\}$,
	\item $T_{10}=\big(H_{\alpha_1}\times F_6\big)\phi=Q_{E^*}(X)\setminus\{\alpha_{6},\alpha_{12},\alpha_{18},\alpha_{24},\alpha_{30},\alpha_{36}\}$.
\end{enumerate}

%%%%%%%%%%%%%%%%%%%%%%%%%%End%%%%%%%%%%%%%%%%%%%%%%%%%%%%%%%%%%%%%%%%%%%%%%%%%%%%%%%%%%%%%%%%%%%%%%%%%%%%%%%%%%%%%%%%%%%%%%

\subsection*{Acknowledgments.} This research was supported by Chiang Mai University.

\bibliographystyle{abbrv}\addcontentsline{toc}{section}{References}
\bibliography{References}

\begin{flushleft}
\vskip.3in

KRITSADA SANGKHANAN, Department of Mathematics, Faculty of Science, Chiang Mai University, Chiang Mai, 50200, Thailand; e-mail: kritsada.s@cmu.ac.th

UTSITHON CHAICHOMPOO, Doctor of Philosophy Program in Mathematics, Department of Mathematics, Faculty of Science, Chiang Mai University, Chiang Mai, 50200, Thailand; e-mail: flash.ex@hotmail.com

\end{flushleft}
\end{document}